\documentclass[reqno,11pt, letterpaper]{amsart}
\PassOptionsToPackage{fleqn}{amsmath}
\RequirePackage{amsthm}
\RequirePackage{amsmath}
\RequirePackage{latexsym}
\RequirePackage{amssymb}
\RequirePackage{amscd}
\RequirePackage{epsfig}
\RequirePackage{graphics}
\RequirePackage{ifthen}
\RequirePackage{varioref}
\usepackage{changes}
\usepackage{soul}
\usepackage{xcolor}
\usepackage{color}
\usepackage[misc]{ifsym}
\usepackage{bm}
\usepackage{lipsum}

\usepackage{epstopdf}

\definecolor{darkgreen}{rgb}{0.0625,0.64,0.0625}
\ifpdf
  \usepackage[
    pdftex,
    colorlinks,%
    linkcolor=blue,citecolor=blue,urlcolor=blue,
    hyperindex,%
    plainpages=false,%
    bookmarksopen,%
    bookmarksnumbered%
  ]{hyperref}
\else
  \usepackage{hyperref}
\fi
\usepackage{bookmark}
\allowdisplaybreaks[1]

\def\R{{\mathbb R}}

\theoremstyle{plain}
\newtheorem{thm}{Theorem}[section]
\newtheorem{lem}[thm]{Lemma}
\newtheorem{prop}[thm]{Proposition}
\newtheorem{cor}[thm]{Corollary}
\theoremstyle{definition}
\newtheorem{rem}[thm]{Remark}
\newtheorem{exmp}{Example}[section]

\def\eq#1{{\rm(\ref{#1})}}
\def\ro#1{{\rm #1}}

\def\Bbb#1{{\mathbb#1}}

\def\R{\Bbb R}

\def\Rx{\R\mkern1mu^}

\def\Rn{\Rx n}

\def\AG#1{\ro{A}(#1)}  
\def\AGR#1{\AG{#1,\R}} 
\def\SA#1{\ro{SA}(#1)} 
\def\SAR#1{\SA{#1,\R}} 
\def\GL#1{\ro{GL}(#1)} 
\def\GLR#1{\GL{#1,\R}} 
\def\SL#1{\ro{SL}(#1)} 
\def\SLR#1{\SL{#1,\R}}  

\def\semidirect{\ltimes}

\numberwithin{equation}{section}
\newcommand\blfootnote[1]{%
 \begingroup
 \renewcommand\thefootnote{}\footnote{#1}%
 \addtocounter{footnote}{-1}%
 \endgroup
 }

\begin{document}

\title[A nonlocal curve flow in centro-affine geometry]{A nonlocal curve flow \\ in centro-affine geometry}

    \author[X.J. Jiang]{Xinjie Jiang}
    \address{ Xinjie Jiang\newline\indent
     Department of Mathematics, Northeastern University, Shenyang, 110819, P.R. China}
    \email{serge0912@icloud.com}

    \author[Y. Yang]{Yun Yang}
    \address{ Yun Yang\newline\indent
     Department of Mathematics, Northeastern University, Shenyang, 110819, P.R. China}
    \email{yangyun@mail.neu.edu.cn}

    \author[Y.H. Yu]{Yanhua Yu${}^*$} \blfootnote{${}^*$~Corresponding author: yuyanhua@mail.neu.edu.cn}
    \address{Yanhua Yu\newline\indent
     Department of Mathematics, Northeastern University, Shenyang, 110819, P.R. China}
    \email[Corresponding author]{yuyanhua@mail.neu.edu.cn}

\begin{abstract}
In this paper, the isoperimetric inequality in centro-affine plane geometry is obtained. We also investigate the long-term behaviour for an invariant plane curve flow, whose evolution process can be expressed as a second-order nonlinear parabolic equation  with respect to centro-affine curvature. The forward and backward limits in time are discussed,
which shows that a closed convex embedded curve may converge to an ellipse when evolving according to this flow.
\end{abstract}




\maketitle
\section{Introduction}
The study of {\it affine differential geometry} relies on  the Lie group $\AGR n = \GLR n \semidirect \Rn$ consisting of affine transformations $x \longmapsto Ax+b$, $A\in \GLR n$, $b\in \Rn$ acting on $x \in \Rn$
(see Nomizu and Sasaki \cite{ns} and Simon \cite{sim} for details), and correspondingly {\it equi-affine geometry} is restricted to the subgroup $\SAR n = \SLR n \semidirect \Rn$ of volume-preserving affine transformations.
In affine differential geometry, a crucial issue  is to investigate the resulting invariants associated with submanifolds $M \subset \Rn$.
Historically, the classical theory of equi-affine hypersurfaces was developed by Blaschke and his collaborators \cite {Blaschke}.
\emph{Centro-affine differential geometry} refers to the subgroup of the affine transformation group that keeps the origin fixed,  which is closely related to the geometry induced
by the general linear group $x \longmapsto Ax$, $A\in \GLR n$, $x \in \Rn$.
Strictly speaking, \emph{centro-equiaffine differential geometry} arises in connection with the subgroup $\SLR n$ of volume-preserving linear transformations.
These cases are usually mentioned in books devoted to (equi-)affine geometry \cite{ns}.

Let $G$ be a Lie group acting on a smooth manifold $M$. An invariant submanifold flow (cf. \cite{olv-isf}) is a $G$-invariant evolutionary partial differential equation
$\displaystyle\frac{\partial S}{\partial t} = \Phi[S],$
which governs the motion of submanifold $S\subset M$.
For the invariant submanifold flow, a fundamental issue is to determine the induced evolution of the geometric invariants characterizing the submanifold $S$.
In the past several decades, considerable research has been devoted to invariant geometric flows for curves and surfaces in Euclidean and affine geometries. Among them, the simplest and well-studied one is the so-called curve shortening flow (CSF). CSF was introduced firstly by Mullins \cite{mul} as a model for the motion of grain boundaries, and later Gage and Hamilton \cite{gh} proved  that a convex curve embedded in ${\mathbb R}^2$ shrinks to a point  when evolving under CSF.  In the sequel, Grayson \cite{gra-1,gra-2} studied the evolution of non-convex embedded curves and confirmed that if the initial curve is any embedded curve in ${\mathbb R}^2$, then the corresponding curve first becomes convex and finally shrinks to a point in finite time while becoming asymptotically circular, often referred to as a ``circular point''.

The corresponding affine curve shortening flow (ACSF) was firstly studied by  Sapiro and Tannenbaum \cite{st}, and was further investigated by Angenent, Sapiro, and Tannenbaum \cite{ast}. It was shown  that a closed convex curve, when evolved according to ACSF, will shrink to an embedded close curve and finally to a point with absolute curvature converging to $2\pi$ \cite{ast,cz}.
Andrews \cite{and0} extended the results of the affine curve flow to the affine hypersurface flow and then studied an affine-geometric, fourth-order parabolic evolution equation for closed convex curves in the plane \cite{and}, which shows the evolving curve remains strictly convex while expanding to infinite size and approaching a
homothetically expanding ellipse.
Loftin and Tsui discussed the noncompact solutions to the affine normal
flow of hypersurfaces and asserted that all ancient solutions must be either ellipsoids (shrinking solitons) or paraboloids
(translating solitons) \cite{lt}.
Chen classified convex compact ancient solutions to the affine curve shortening flow, namely, any convex compact ancient solution to the affine curve shortening flow must be a shrinking ellipse \cite{chen}.
There are also a series of achievements in centro-equiaffine geometric flow obtained by Ivaki and Stancu. See \cite{is, iva0, iva} and references therein.
In computer vision  \cite{olv-6, ost-1,ost-2}, Euclidean curve shortening flow and its equi-affine counterpart have been successfully applied to image denoising and segmentation, they were also actively used in practical computer implementations. More recently, similar results were obtained for the heat flow in  centro-equiaffine geometry \cite{wwq} and  centro-affine geometry \cite{oqy}.  Heat flows in more general Klein geometries were investigated \cite{olv-isf, ost-1}.

For a given Lie group $G$ or a Lie algebra $\mathfrak{g}$, one can define its group curvature as the differential invariant of the
lowest order and group arc-length element $ds$ as the invariant one-form. The group invariant geometric heat flow is simply given by
\begin{equation}\label{ighf}
  \frac{\partial C}{\partial t} = C_{ss}.
\end{equation}

When the group $G$ is the Euclidean group,  the flow \eqref{ighf} is CSF (cf. \cite{gh}) and yields a second-order nonlinear  parabolic equation for the Euclidean curvature $\kappa$ and Euclidean arc-length $s$
\begin{equation*}
\frac{\partial\kappa}{\partial t}=\frac{\partial^2\kappa}{\partial s^2}+\kappa^3.
\end{equation*}

When the group $G$ is the special affine group,  the flow \eqref{ighf} is called ACSF. By \cite{ast,st}, this flow generates a second-order nonlinear  parabolic equation for the equi-affine curvature $\kappa$ and equi-affine arc-length $s$
\begin{equation*}
\frac{\partial\kappa}{\partial t}=\frac{4}{3}\frac{\partial^2\kappa}{\partial s^2}+\frac{1}{3}\kappa^2.
\end{equation*}

When the group $G$ is the centro-equiaffine group, by \cite{wwq}, the flow \eqref{ighf} induces a second-order nonlinear  parabolic equation for the centro-equiaffine curvature $\kappa$ and centro-equiaffine arc-length $s$
\begin{equation*}
\frac{\partial\kappa}{\partial t}=\frac{\partial^2\kappa}{\partial s^2}+4\kappa^2.
\end{equation*}

In \cite{oqy}, the centro-affine invariant geometric heat flow was investigated, that is, $G$ is the general linear group. In this case, the flow \eqref{ighf} leads to the well-known inviscid Burgers' equation with respect to the centro-affine curvature $\kappa$ and centro-affine arc-length $s$
\begin{equation*}
\frac{\partial\kappa}{\partial t}=\kappa\kappa_s.
\end{equation*}

Therefore, an interesting problem arises: Can we find a flow that yields a second-order nonlinear  parabolic equation for the centro-affine curvature and has some  analogous properties to the invariant heat flows in Euclidean, equi-affine and centro-equiaffine geometry.

Note that, from now on, we use $\xi$  and $\varphi$ to  denote the centro-affine arc-length parameter and centro-affine curvature, respectively. In \cite{qy},
the curve flow in the centro-affine setting
\begin{equation}\label{pC_t}
	\frac{\partial C(p,t)}{\partial t}=\left(\lambda+\int_{0}^{\xi(p)} \varphi d\xi\right)C+\frac{\varphi}{2}C_\xi, \quad C(\cdot,0)=C_0(\cdot),
\end{equation}
was studied, where $ \lambda $ is a constant. This flow generates a second-order nonlinear  parabolic equation for the centro-affine curvature $\varphi$
\begin{equation*}
	\frac{\partial \varphi}{\partial t}=\frac{1}{2}\varphi_{\xi\xi}-\frac{1}{2}\varphi^3+2\varphi.
\end{equation*}
The choice of constant $ \lambda $ does not influence the whole flow. In fact, by a reparametrization $\displaystyle \tilde{C}=\exp(-\lambda t)C$, the flow is centro-affine equivalent to
	\begin{align*}
		\frac{\partial \tilde{C}(p,~t)}{\partial t}=\Big(\int^{\xi(p)}_0\varphi d\xi\Big) \tilde{C}+\frac{\varphi}{2}\tilde{C}_{\xi},\quad
		\tilde{C}(\cdot,0)=C_0(\cdot).
	\end{align*}

In the current paper, we obtain a centro-affine isoperimetric inequality
(see Proposition \ref{dz} for details)
\begin{equation*}
L=\oint d\xi\leq2\pi
\end{equation*}
for any convex smooth embedded closed curve in $\R^2$, and equality holds only for the ellipse centered at the origin.
 In particular, the centro-affine isoperimetric inequality is essential  in the
proof of Theorem \ref{thm}.

The local existence and uniqueness of this flow has been discussed in \cite{qy},
and in this paper we will prove that the solution to \eqref{pC_t} continues as $t\to+\infty$ (see Theorem \ref{thm-globel} for details).
Nonlocal flows have been widely used to prove geometric inequalities in the Euclidean and hyperbolic space \cite{and1,ath,cm,is,mc1,mc2,mc3,sin},
and this method has become very popular during past the decades \cite{gpt,gps,sch},
in which the long-time existence of a suitably defined geometric flow, convergence to a steady state and monotonicity of geometric quantity play a pivotal role.
In this paper, we mainly focus on the following two theorems.
\begin{thm}\label{thm}
Assume $ C(\cdot,t) $ is a solution of \eqref{pC_t} in a maximal interval $ [0,+\infty)  $, where $ C_0 $ is a closed convex smooth embedded curve. Then the centro-affine curvature $ \varphi $ of $ C(\cdot,t) $ converges smoothly to zero, that is, $ C(\cdot,t) $ converges to an ellipse as $t\to+\infty$.
\end{thm}

\begin{thm}\label{thm2}
Assume the closed convex smooth embedded curve family $C(\cdot,t) $ is a solution of \eqref{pC_t} in the interval $ (-\infty,+\infty)$ with  uniformly bounded centro-affine curvature. Then the centro-affine curvature $ \varphi $ of $ C(\cdot,t) $ converges smoothly to zero as $t\to-\infty$, that is, $ C(\cdot,t) $ converges to an ellipse
as $t\to-\infty$.
\end{thm}

This paper is organized as follows. In Section 2, we recall the relevant definitions, notions,  and basic facts for curves in centro-equiaffine and centro-affine geometries. In particular, the centro-affine isoperimetric inequality is obtained. Section 3 introduces the centro-aﬀine invariant curve flow \eqref{pC_t} and shows that it is nonlocal. Some energy inequalities are derived in Section 4. Section 5 combines the results of energy estimates to verify that a closed embedded curve may converge to an ellipse when evolving according to this flow. In Sections 6 and 7, we discuss the backward limit of the  centro-affine curvature of this flow and complete the proof of Theorem \ref{thm2}.

\section{Preliminaries }\label{sec-back}
 Let $ C(p): I\rightarrow\mathbb{R}^2 $ be an embedded plane curve with parameter $ p $, where $ I\subset \R$ denotes a certain interval. In this paper, we will always assume the mappings are  sufficiently smooth so that all the relevant derivatives are well-defined and focus our attention on {\it regular star-shaped} centro-affine curves that satisfy $ \displaystyle[C_i,C_{i+1}]\neq0 $, where the bracket notation stands for the determinant,  $ C_0=C $ and  $\displaystyle C_i=\frac{d^iC}{dp^i}$  for $ i>1. $
\subsection{Centro-affine differential geometry}
We start by reviewing  some  of the relevant basic  concepts   about centro-affine geometry.
The \textit{centro-affine transformation}  in the plane is defined as
\begin{equation*}
 \mathbf{x} \longmapsto \mathcal{A}\mathbf{x},
\end{equation*}
where $\mathcal{A}\in\GL{2,\mathbb{R}}$ and $ \textbf{x}\in\mathbb{R}^2 $. Although our interest lies  in centro-affine geometry, let us talk briefly about centro-equiaffine geometry, i.e., restrict $ \mathcal{A} $ in the  subgroup $ \SLR{2} $.
 The \textit{centro-equiaffine arc-length}, invariant under the centro-equiaffine transformation, is defined as  a parameter $ \sigma $ such that
\begin{equation}\label{arc}
	[C,C_\sigma]=1.
\end{equation}
It naturally follows that
\begin{equation*}
	\frac{d\sigma}{dp}=[C,C_p],
\end{equation*}
where $ p $ is the free parameter.
Differentiating both sides of the equation \eqref{arc} with respect to the centro-equiaffine arc-length parameter $\sigma$ yields
\begin{equation*}
	[C,C_{\sigma\sigma}]=0,
\end{equation*}
which implies that $ C $ and $ C_{\sigma\sigma} $ are not independent.
The \textit{centro-equiaffine curvature} $ \mu  $, the only generating differential invariant of plane curves, is defined as a function such that
\begin{equation*}
	\mu C+C_{\sigma\sigma}=0,
\end{equation*}
or, equivalently,
\begin{equation*}
	\mu=[C_\sigma,C_{\sigma\sigma}].
\end{equation*}
Now let us discuss the centro-affine geometry. In \cite{oqy}, the \textit{centro-affine arc-length} $ \xi $ is a parameter determined from the relation
\begin{equation}\label{arc-c}
	\left\vert \frac{[C_\xi,C_{\xi\xi}]}{[C,C_\xi]} \right\vert=1.
\end{equation}
If we introduce the notation
\begin{equation*}\label{sgn}
	\epsilon:={\rm sgn}\left(\frac{[C_p,C_{pp}]}{[C,C_p]}\right),
\end{equation*}
where $ {\rm sgn} (\cdot)$ is the signum function, then the equation \eqref{arc-c} immediately leads to
\begin{equation}\label{ca-arc}
	\frac{d\xi}{dp}=\sqrt{\epsilon\frac{[C_p,C_{pp}]}{[C,C_p]}}.
\end{equation}
Moreover, the {\it centro-affine curvature} $\varphi$, the lowest order centro-affine differential invariant, is also  defined in \cite{oqy}
\begin{equation*}\label{cur}
	\varphi=\frac{[C_{\xi\xi},C]}{[C_\xi,C]}=\sqrt{\epsilon\frac{[C,C_p]}{[C_p,C_{pp}]}}\left(\frac{3}{2}\frac{[C,C_{pp}]}{[C,C_p]}-\frac{1}{2}\frac{[C_p,C_{ppp}]}{[C_p,C_{pp}]}\right).
\end{equation*}
The following equations
\begin{equation}\label{rel-eca-ca}
	\epsilon={\rm sgn}(\mu),\qquad \frac{d\xi}{d\sigma}=\sqrt{\epsilon\mu},\qquad \varphi=-\frac{\epsilon}{2}(\epsilon\mu)^{-3/2}\mu_\sigma
\end{equation}
reveal the relation between the invariants of centro-equiaffine geometry and centro-affine geometry, which can be obtained by considering parameter $ p $ in the corresponding equation as centro-equiaffine arc-length parameter $ \sigma $.

\subsection{Centro-affine isoperimetric inequality}
For the closed curve $ C $ in the centro-affine setting,  we have the following two results (see \cite{qy} for details).
\begin{lem}\label{meanzero}
	If the curve C is closed, then $ \displaystyle\oint_{C} \varphi d\xi =0$.
\end{lem}
\begin{lem}\label{lem-ep-1}
	If the curve C is closed, then $\epsilon$=1 and the origin lies inside every simple closed loop of the curve C.
\end{lem}
Now let us  derive the centro-affine isoperimetric inequality.
\begin{prop}\label{dz}
	In centro-affine geometry, for any smooth convex  embedded closed curve $C$, the centro-affine perimeter of the curve
\begin{equation*}\label{iso-ca}
	L=\oint_C d\xi\le2\pi,
\end{equation*}
and equality holds if and only if $ C $ is an ellipse centered at the origin.
\end{prop}
\begin{proof}
	Let us first express the centro-affine perimeter of the curve in centro-equiaffine setting. By \eqref{rel-eca-ca} and Lemma \ref{lem-ep-1},
	\begin{equation*}
		L=\oint_Cd\xi=\oint_C\sqrt{\mu}d\sigma.
	\end{equation*}
Then we will use the isoperimetric inequality in centro-equiaffine plane geometry \cite{wwq}, which states that
	\begin{equation}\label{4}
		A\oint_C\mu d\sigma\le2\pi^2,
	\end{equation}
	where $ A $ is the area enclosed by the curve, and equality holds if and only if $C$ is an ellipse. It is worth emphasizing that in the cetro-equiaffine setting,
	\begin{equation*}
		A=\frac{1}{2}\oint_C[C,C_\sigma]d\sigma=\frac{1}{2}\oint_Cd\sigma.
	\end{equation*}
Substituting back into \eqref{4} yields
\begin{equation*}
	\oint_Cd\sigma\oint_C\mu d\sigma\le4\pi^2.
\end{equation*}
Hence, employing the Cauchy-Schwarz inequality, we deduce that
	\begin{equation*}
		\left(\oint_{C}d\xi\right)^2=\left(\oint_C \sqrt{\mu} d \sigma\right)^2\le\oint_Cd\sigma\oint_C\mu d\sigma\le 4\pi^2,
	\end{equation*}
	which implies
	\begin{equation*}
		\oint_Cd\xi\le2\pi.
	\end{equation*}
	To show the equality holds if and only if $ C $ is an ellipse, we need only to notice the following fact.
    If $C$ is an ellipse expressed by $(x_0+a\cos p, y_0+b\sin p)^{\mathbf{T}},\; p\in[0, 2\pi)$. Here the superscript ``${}^\mathbf{T}$" represents the transpose of a
    vector. By \eqref{ca-arc}, it is easy to check
    \begin{equation*}
    \oint_C d\xi =2\pi
    \end{equation*}
    only holds for $x_0=0$ and $y_0=0$.
\end{proof}

\section{A second-order curve flow in centro-affine geometry}
Let us now investigate the centro-affine invariant curve flow, described by a family of time-dependent embedded smooth closed curves  $ C(p,t): S^1\times I\rightarrow\mathbb{R}^2 $ satisfying the following evolution equation
\begin{equation*}
	\frac{\partial C(p,t)}{\partial t}=\left(\lambda+\int_{0}^{\xi(p)} \varphi d\xi\right)C+\frac{\varphi}{2}C_\xi
\end{equation*}
with the initial condition
\begin{equation*}
	C(\cdot,0)=C_0(\cdot),
\end{equation*}
where  $ \lambda $ is a constant.
In view of \eqref{ca-arc} and Lemma \ref{lem-ep-1}, the invariant centro-affine metric for the closed curve $ C(p,t) $ can be represented as
\begin{equation*}
	g(p):=\sqrt{\frac{[C_p,C_{pp}]}{[C,C_p]}}
\end{equation*}
so that the arc-length parameter $ \xi $ is defined by
\begin{equation*}
	\xi(p)=\int_{p_0}^p g(p)dp.
\end{equation*}
According to the relation $\displaystyle \frac{\partial^2}{\partial t\partial p}=\frac{\partial^2}{\partial p\partial t}$, the commutator of $\displaystyle \frac{\partial}{\partial\xi} $ and $\displaystyle \frac{\partial}{\partial t} $ is readily computed
\begin{equation}\label{huanwei}
	\frac{\partial^2}{\partial t\partial \xi}-\frac{\partial^2}{\partial \xi\partial t}=-\frac{1}{g}\frac{\partial g}{\partial t}\frac{\partial}{\partial\xi}.
\end{equation}
The evolution equations for centro-affine metric and curvature have been derived in  \cite{oqy,qy}
\begin{align}
	\frac{1}{g}\frac{\partial g}{\partial t}&=\frac{1}{2}\varphi ^2,\label{g_t}\\
	 \frac{\partial \varphi}{\partial t}&=\frac{1}{2}\varphi_{\xi\xi}-\frac{1}{2}\varphi^3+2\varphi.\label{phi_t}
\end{align}
The following proposition  states that we can investigate the centro-affine curve flow by studying the evolution equation for its centro-affine curvature.
\begin{prop}[\cite{qy}] There exists a time $T$ such that for $t\in[0,T)$, the centro-affine curve evolution process
	\begin{equation*}\label{Curve-ev}
		\frac{\partial C}{\partial t}=\Big(\int_{0}^{\xi}\varphi d\xi\Big)C+\frac{\varphi}{2}C_\xi,\qquad C(\cdot,0)=C_0(\cdot)
	\end{equation*}
	is equivalent to finding a solution to the PDE problem for $\varphi\in C^{3+\alpha,1+\alpha}(S^1\times[0,T))$ such that
	\begin{equation}\label{PDE-pro}
		\frac{\partial \varphi}{\partial t}=\frac{1}{2}\varphi_{\xi\xi}-\frac{1}{2}\varphi^3+2\varphi,\qquad
		\varphi(0)=\varphi(C_0)=\varphi_0,
	\end{equation}
	where $C_0(\cdot)$ is smooth and closed, and $\varphi_0$ is bounded, smooth and periodic.
\end{prop}
The short time existence of \eqref{PDE-pro} is evident (refer to  \cite{qy} for details),  and  the unique solvability of the flow \eqref{pC_t} for a small time has been studied in \cite{qy}. Let us next show  that this solution exists on $ [0,+\infty) $.
It is worth noting that Lemma \ref{meanzero} tells us that the centro-affine curvature of a closed curve is a periodic function with zero mean.  We thus conclude that
\begin{equation*}
	\varphi_{\max}(t):=\sup_{p}\{\varphi(p,t)\}\geq0, \qquad
	\varphi_{\min}(t):=\inf_{p}\{\varphi(p,t)\}\leq0
\end{equation*}
for $t \in [0, T )$.
From this, we can prove that
\begin{lem}\label{lem-bd1}
	If the centro-affine curvature $\varphi$ satisfies \eq{phi_t}  for $t \in [0, T )$, then  $\varphi$ is bounded in $[0, T )$. Furthermore, $\min\{-2,\varphi_{\min}(0)\}\leq\varphi(p,t)\leq\max\{2,\varphi_{\max}(0)\}$.
\end{lem}
\begin{proof}
	Let us show that  $ \varphi(p,t)  $ has a lower bound. If $ \varphi(p,t) \ge\varphi_{\min}(0) $, the conclusion clearly holds. Otherwise, there exists  some $t>0$ such that  $\beta:=\varphi_{\min}(t)<\varphi_{\min}(0)\leq0$.  Suppose that  $\beta$ is achieved for the first time at
	$(p_0, t_0)$,  i.e., $ \varphi(p_0, t_0)=\beta $ and $\displaystyle t_0=\inf\{t:\varphi(t)=\beta\}.$ Then we can see that  $\varphi_t\leq0,$ $\varphi_{pp}\geq0$, $\varphi_p=0$ at $(p_0, t_0)$.
	Further,
	\begin{equation*}
		\varphi_{\xi\xi}=\varphi_{pp}\Big(\frac{dp}{d\xi}\Big)^2+\varphi_p\frac{d^2p}{d\xi^2}\geq0.
	\end{equation*}
Hence, it follows from  \eq{phi_t} that $-\frac{1}{2}\beta(\beta^2-4)\leq0$, which leads to $\beta\geq-2$.
The proof that  $ \varphi(p,t)  $ has an upper  bound can be obtained by a similar argument.
\end{proof}

\begin{lem}\label{lem-bd2}
	If $\varphi$ satisfies \eq{phi_t}  for $t \in [0, T )$, then $\displaystyle\frac{\partial^n \varphi}{\partial \xi^n} (n\geq0)$ are bounded in the interval.
\end{lem}
\begin{proof}
	First, by \eqref{phi_t} and \eqref{huanwei}, the evolution equation for $\displaystyle\frac{\partial\varphi}{\partial\xi}$ can be easily computed:
	\begin{equation}\label{xi_t}
	\frac{\partial}{\partial t}\Big(\frac{\partial\varphi}{\partial\xi}\Big)
	=\frac{1}{2}\frac{\partial^2}{\partial\xi^2}\Big(\frac{\partial\varphi}{\partial\xi}\Big)
	+(2-2\varphi^2)\frac{\partial\varphi}{\partial\xi}.
\end{equation}
	Applying the standard theory on parabolic equations.
	The above equation bounds the rate of growth of $\displaystyle\frac{\partial\varphi}{\partial\xi}$ to exponential.  Therefore,  $\displaystyle\frac{\partial\varphi}{\partial\xi}$ is bounded for finite time. In general, a similar computation gives
	\begin{equation*}
		\frac{\partial}{\partial t}\Big(\frac{\partial^n\varphi}{\partial\xi^n}\Big)
		=\frac{1}{2}\frac{\partial^2}{\partial\xi^2}\Big(\frac{\partial^n\varphi}{\partial\xi^n}\Big)
		+\Big(2-\frac{n+3}{2}\varphi^2\Big)\frac{\partial^n\varphi}{\partial\xi^n}+\text{previously bounded terms}.
	\end{equation*}
	Thus, $\displaystyle\frac{\partial^n\varphi}{\partial\xi^n}$ grows not faster than
	exponentially, and it is bounded for finite time.
\end{proof}
We now prove the following theorem, which demonstrates that the flow
\eqref{pC_t} is nonlocal.
\begin{thm}\label{thm-globel} The solution to the PDE problem \eq{PDE-pro} (and so to the centro-affine curve evolution process) continues as $t\to+\infty$.
\end{thm}
\begin{proof}
	By Lemma \ref{lem-bd1} and Lemma \ref{lem-bd2}, $\displaystyle\varphi$ and all its derivatives $\displaystyle\frac{\partial^n\varphi}{\partial\xi^n}$ are bounded. Using \eq{phi_t},
	bounds on the time derivatives can also be obtained.
	Thanks to the short-term results, we can assume the solution to \eqref{PDE-pro} exists on the interval  $ [0,T) $. Then $\varphi(t)$
	has a limit as $t\to T$, and the limiting curve $C(T)$ is smooth. Again, from the short time results, the solution $\varphi(t)$ and the corresponding  curve $C(t)$
	exist for some interval $[T,T+\epsilon),\;\epsilon>0$. This process can be continued.
\end{proof}

\section{Energy estimates}\label{EE}
In this section, we will derive some energy inequalities for the flow  \eqref{pC_t}.  The energy here  is defined as
 \begin{equation*}
 	E=\displaystyle\oint_C\varphi^2(\xi,t)d\xi.
 \end{equation*}
Before proceeding further, let us first review the interpolation inequality, which will be repeatedly used  relying on the fact in Lemma \ref{meanzero} that $\displaystyle\oint_C\varphi d\xi=0 $.
It states that for periodic function $ u $ with zero mean, the following inequality
\begin{equation*}\label{cz}
	\big\|u^{(j)}\big\|_{L^r}\le c\big\|u\big\|_{L^p}^{1-\theta}\cdot\big\|u^{(k)}\big\|_{L^q}^\theta, \quad \theta\in(0,1)
\end{equation*}
is valid, where $ j,k,p,q $ and $ r$  satisfy $ p,q,r>1,\; j\ge0,  $
\begin{equation*}
	\frac{1}{r}=j+\theta\left(\frac{1}{q}-k\right)+(1-\theta)\frac{1}{p}, \qquad \frac{j}{k}\le \theta\le 1,
\end{equation*}
and the constant $ c $ depends on $ j,k,p,q $  and $ r $ only.

Let us now get to the point.
For brevity, we employ the more compact notation  $ \displaystyle\varphi_{\xi^n} $ to indicate $\displaystyle\frac{\partial^n\varphi}{\partial\xi^n} $ for every integer $ n\ge0 $.
Firstly, by \eqref{g_t} and \eqref{phi_t}, it is easy to compute directly that
\begin{equation*}\label{0j}
	\frac{d}{dt}E=-\oint_C\varphi_\xi^2d\xi-\frac{1}{2}\oint_C\varphi^4d\xi+4E.
\end{equation*}
Then, applying the interpolation inequality to estimate the second term on the right-hand, we deduce that
\begin{equation*}\label{nl}
	\oint_C\varphi^4d\xi  \le cE^\frac{3}{2}\left(\oint_C\varphi_\xi^2 d\xi\right)^\frac{1}{2}\le
	 \epsilon\oint_C\varphi_\xi^2d\xi+\frac{1}{2}(2\epsilon)^{-1}c^2E^3,
\end{equation*}
from which it follows that
\begin{equation*}
	\frac{d}{dt}E\le\left(-1+\frac{1}{2}\epsilon\right)\oint_C\varphi_\xi^2d\xi+p(E),  \quad \text{where} \quad	p(E) = 4E+\frac{1}{8\epsilon}c^2E^3.
\end{equation*}
Thus,  choosing $ \epsilon=2 $, we conclude that
\begin{equation}\label{0jgj}
	\frac{d}{dt}E\le c_1(E+E^3),
\end{equation}
where $ c_1 $ is a constant.

Further,  the  more general inequality can  be established by a similar procedure. To do this, we introduce the notation  $ \displaystyle\varphi_{\xi^{i_1}}*\varphi_{\xi^{i_2}}*\cdot\cdot\cdot*\varphi_{\xi^{i_m}} $, where $\displaystyle i_1\le i_2\le \cdot\cdot\cdot\le i_m, $ to denote the term  $\displaystyle K\varphi_{\xi^{i_1}}\varphi_{\xi^{i_2}}\cdot\cdot\cdot\varphi_{\xi^{i_m}} $ in which the specific value of non-zero constant $ K $ is not required.  Employing this notation, we can rewrite the previous evolution equation in the form
\begin{align}\label{phit}
	\varphi_t=\frac{1}{2}\varphi_{\xi\xi}+\varphi*\varphi*\varphi+2\varphi,\qquad
	\varphi_{\xi t}=\frac{1}{2}\varphi_{\xi^3}+\varphi*\varphi*\varphi_\xi+2\varphi_\xi.
\end{align}
Now, let us generalize the inequality \eqref{0jgj}. Indeed, the general energy inequality reads as follows.
\begin{prop}\label{nee}
	For every nonnegative integer $n$, we have
	\begin{equation}\label{varphint}
		\frac{d}{dt}\oint_{C}\varphi_{\xi^n}^2d\xi \le c(E+E^{2n+3}),
	\end{equation}   where $ c $ is a positive constant.
\end{prop}
The proof of Proposition \ref{nee} relies on the following three elementary lemmas.
\begin{lem}\label{xint}
	 For an arbitrary nonnegative integer $n$, there holds
\begin{equation*}
	\varphi_{\xi^n t}=\frac{1}{2}\varphi_{\xi^{n+2}}+\sum_{i_1+i_2+i_3=n}^{}\varphi_{\xi^{i_1}}*\varphi_{\xi^{i_2}}*\varphi_{\xi^{i_3}}+2\varphi_{\xi^n}.
\end{equation*}
\end{lem}
\begin{proof}
	The proof is by complete induction on $ n $. Obviously, it follows from \eqref{phit} that the equation holds for $ n=0 $.
	Assume that this holds for some $ n\ge0 $, then a direct computation gives
	\begin{align*}
		\varphi_{\xi^{n+1}t} & =\varphi_{\xi^n t \xi}-\frac{g_t}{g}\varphi_{\xi^{n+1}} \\ & =\frac{\partial}{\partial\xi}\left(\frac{1}{2}\varphi_{\xi^{n+2}}+\sum_{i_1+i_2+i_3=n}^{}\varphi_{\xi^{i_1}}*\varphi_{\xi^{i_2}}*\varphi_{\xi^{i_3}}+2\varphi_{\xi^n}\right)-\frac{1}{2}\varphi^2\varphi_{\xi^{n+1}} \\ & =\frac{1}{2}\varphi_{\xi^{n+3}}+\sum_{j_1+j_2+j_3=n+1}^{}\varphi_{\xi^{j_1}}*\varphi_{\xi^{j_2}}*\varphi_{\xi^{j_3}}+2\varphi_{\xi^{n+1}},
	\end{align*}
	which completes the inductive step, hence our proof of this proposition.
\end{proof}

\begin{lem}\label{zk}For every integer $n\ge0$, we have
	\begin{align*}
		\frac{d}{dt} \oint_{C}\varphi_{\xi^n}^2d\xi=&-\oint_C\varphi_{\xi^{n+1}}^2d\xi+\sum_{i_1+i_2+i_3=n}^{}\oint_C\varphi_{\xi^{i_1}}*\varphi_{\xi^{i_2}}*\varphi_{\xi^{i_3}}*\varphi_{\xi^n}d\xi\\
            &\qquad\qquad\qquad\qquad\qquad\qquad\qquad\qquad+4\oint_C\varphi_{\xi^n}^2d\xi.
	\end{align*}
\end{lem}
\begin{proof}
	Applying Lemma \ref{xint} and integration by parts, it is not hard to see
	\begin{align*}
		\frac{d}{dt}\oint_{C}\varphi_{\xi^n}^2d\xi  =&2\oint_C\varphi_{\xi^n}\varphi_{\xi^nt}d\xi+\oint_C\varphi_{\xi^n}^2\frac{g_t}{g}d\xi \\=&\oint_C\varphi_{\xi^n}\varphi_{\xi^{n+2}}d\xi+\sum_{i_1+i_2+i_3=n}^{}\oint_C\varphi_{\xi^{i_1}}*\varphi_{\xi^{i_2}}*\varphi_{\xi^{i_3}}*\varphi_{\xi^n}d\xi\\
&\qquad\qquad\qquad\qquad\qquad\qquad+4\oint_C\varphi_{\xi^n}^2d\xi+\frac{1}{2}\varphi^2\varphi_{\xi^n}^2d\xi\\ =&-\oint_C\varphi_{\xi^{n+1}}^2d\xi+\sum_{i_1+i_2+i_3=n}^{}\oint_C\varphi_{\xi^{i_1}}*\varphi_{\xi^{i_2}}*\varphi_{\xi^{i_3}}*\varphi_{\xi^n}d\xi\\&
\qquad\qquad\qquad\qquad\qquad\qquad+4\oint_C\varphi_{\xi^n}^2d\xi.
	\end{align*}
	Hence we complete the proof.
\end{proof}
\begin{lem}\label{gj}
	For every integer $ m\ge1 $ and $ n \ge 0 $, it follows
	\begin{equation*}
		\oint_C \left|\varphi_{\xi^{i_1}}\varphi_{\xi^{i_2}}\cdot\cdot\cdot\varphi_{\xi^{i_m}}\right| d\xi \le \epsilon\displaystyle\oint_C\varphi_{\xi^k}^2d\xi +c^{\frac{1}{A}}A\left(\frac{4k}{2n+m-2}\epsilon\right)^{B}E^D,
	\end{equation*}
	where $\displaystyle i_1+i_2+\cdot\cdot\cdot+i_m=n
	 $,  $\displaystyle A=\frac{4k-2n-m+2}{4k}$, $\displaystyle B=\frac{2-2n-m}{4k-2n-m+2},$  $ \displaystyle D=\frac{2km-2n-m+2}{4k-2n-m+2}$,  $  \epsilon $ and $ c $ are some constants.
\end{lem}
\begin{proof}
First, employing the H$\rm \ddot{o}$lder inequality produces
	\begin{align}	\label{holder}
		\begin{split}
		\oint_C\left|\varphi_{\xi^{i_1}}\varphi_{\xi^{i_2}}\cdot\cdot\cdot\varphi_{\xi^{i_m}}\right| d\xi&\le\left(\oint_C\left|\varphi_{\xi^{i_1}}\right|^md\xi\right)^{\frac{1}{m}}\cdot\cdot\cdot\left(\oint_C\left|\varphi_{\xi^{i_m}}\right|^md\xi\right)^{\frac{1}{m}}\\
	&=\prod\limits_{j=1}^m\left(\oint_C\left|\varphi_{\xi^{i_j}}\right|^md\xi\right)^{\frac{1}{m}}.
\end{split}
		\end{align}
Then, as before, by interpolation inequality, the terms can be estimated as follows
	\begin{align*}
\left(\oint_C\left|\varphi_{\xi^{i_j}}\right|^md\xi\right)^{\frac{1}{m}}\le c_j E^{\frac{1}{2}\left(1-\frac{i_j}{k}-\frac{1}{2k}+\frac{1}{mk}\right)}\left(\oint_C\varphi_{\xi^k}^2d\xi\right)^{\frac{1}{2}\left(\frac{i_j}{k}+\frac{1}{2k}-\frac{1}{mk}\right)},
	\end{align*}
	where $ c_j $ is a positive constant.
	Thus, substituting back into \eqref{holder}, we obtain
	\begin{align*}
	\oint_C\left|\varphi_{\xi^{i_1}}\cdots\varphi_{\xi^{i_m}}\right| d\xi & \le cE^{\frac{1}{2}\left(m-\frac{n}{k}-\frac{m}{2k}+\frac{1}{k}\right)}\left(\oint_C\varphi_{\xi^k}^2d\xi\right)^{\frac{1}{2}\left(\frac{n}{k}+\frac{m}{2k}-\frac{1}{k}\right)}\\ & \le \epsilon\oint_C\varphi_{\xi^k}^2d\xi +\frac{4k-2n-m+2}{4k}c^{\frac{4k}{4k-2n-m+2}}\\
		&\quad\times\left(\frac{4k\epsilon}{2n+m-2}\right)^{\frac{2-2n-m}{4k-2n-m+2}}E^\frac{2km-2n-m+2}{4k-2n-m+2},
	\end{align*}
	where $ \epsilon $ is an arbitrary positive constant and $ c=c_1c_2\cdots c_m $.
\end{proof}

\begin{proof}[Proof of Proposition \ref{nee}]
	The proof is an immediate consequence of the above two lemmas. Indeed,
	Lemma \ref{zk} implies that
	\begin{align*}
	\frac{d}{dt} \oint_{C}\varphi_{\xi^n}^2d\xi=&-\oint_C\varphi_{\xi^{n+1}}^2d\xi\\
           &\qquad+\sum_{i_1+i_2+i_3=n}^{}\oint_C\varphi_{\xi^{i_1}}*\varphi_{\xi^{i_2}}*\varphi_{\xi^{i_3}}*\varphi_{\xi^n}d\xi+4\oint_C\varphi_{\xi^n}^2d\xi\\ \le&-\oint_C\varphi_{\xi^{n+1}}^2d\xi\\
           &\qquad\quad+K_1\sum_{i_1+i_2+i_3=n}^{}\oint_C\left|\varphi_{\xi^{i_1}}\varphi_{\xi^{i_2}}\varphi_{\xi^{i_3}}\varphi_{\xi^n}\right| d\xi+4\oint_C\varphi_{\xi^n}^2d\xi,
\end{align*}
where $ K_1 $ is some constant.
	Then applying Lemma \ref{gj} leads to
		\begin{align*}
		\oint_C\left|\varphi_{\xi^{i_1}}\varphi_{\xi^{i_2}}\varphi_{\xi^{i_3}}\varphi_{\xi^n}\right| d\xi\le& \epsilon_1\oint\varphi_{\xi^{n+1}}^2d\xi\\&+\frac{1}{2n+2}\left(\frac{2n+2}{2n+1}\epsilon_1\right)^{-(2n+1)}c_1^{2n+2}E^{2n+3},
	\end{align*}
	and
	\begin{equation*}
		\oint_C\varphi_{\xi^n}^2d\xi\le\epsilon_2\oint_C\varphi_{\xi^{n+1}}^2d\xi+\frac{1}{n+1}\left(\frac{n+1}{n}\epsilon_2\right)^{-n}c_2^{n+1}E.
	\end{equation*}
It follows from the former that
	\begin{align*}
	\sum_{i_1+i_2+i_3=n}^{}\oint_C\left|\varphi_{\xi^{i_1}}\varphi_{\xi^{i_2}}\varphi_{\xi^{i_3}}\varphi_{\xi^n}\right| d\xi\le& K_2\epsilon_1\oint\varphi_{\xi^{n+1}}^2d\xi+\frac{K_2}{2n+2}c_1^{2n+2}\\&\qquad\times\left(\frac{2n+2}{2n+1}\epsilon_1\right)^{-(2n+1)}E^{2n+3},
\end{align*}
where $ K_2 $ is some constant. Therefore, by taking appropriate $ \epsilon_1 $ and $ \epsilon_2 $ such that $ K_1K_2\epsilon_1+4\epsilon_2=1 $, we arrive at the desired result.
\end{proof}

\section{Proof of Theorem \ref{thm}}\label{proof-thm}
Before proving the theorem, let  us first, combining the results of the energy estimates, derive some propositions about the energy. We begin with a lemma, which will help with our subsequent argument.
\begin{lem}\label{lem-0toinfty}
	$\displaystyle\int_{0}^{+\infty}E(t)dt$ converges.
\end{lem}
\begin{proof}
	First, a straightforward computation shows that
		\begin{equation*}
		\frac{d}{dt}L(t)=\frac{d}{dt}\oint_{C}d\xi=\oint_{C}\frac{g_t}{g}d\xi=\frac{1}{2}\oint_C\varphi^2d\xi,
	\end{equation*}
that is to say,
	\begin{equation*}
		E(t)=2\frac{d}{dt}L(t).
	\end{equation*}
	Then, integrating both sides from $ 0 $ to $ t $ and employing the centro-affine isoperimetric inequality in Proposition \ref{dz}, we may find
	\begin{equation}\label{4pi}
		\int_{0}^{t} E(\tau)d\tau\le 4\pi.
	\end{equation}
	Hence let $ t\rightarrow +\infty $ yields
	\begin{equation*}
		\int_{0}^{+\infty}E(t)dt\le 4\pi.
	\end{equation*}
	Since $ E(t)\ge0 $, it follows that $  \displaystyle\int_{0}^{t} E(\tau)d\tau$ monotonically increases and has an upper bound. Thus $ 		\displaystyle\int_{0}^{+\infty}E(t)dt$ converges.
\end{proof}
Then, from the centro-affine isoperimetric inequality in Proposition \ref{dz} and the characteristics of the  nonlocal flow in Lemma \ref{lem-bd1} and Theorem \ref{thm-globel}, it can be readily verified that
\begin{prop}\label{eyj}
	The energy $ E(t)=\displaystyle\oint_C\varphi^2(\xi,t)d\xi $ is uniformly bounded on the interval  $ [0,+\infty) $.
\end{prop}
\begin{cor}\label{cor-vn2-bd} For an arbitrary positive integer $n$, $\displaystyle \oint_C\varphi^2_{\xi^n}d\xi$ is uniformly bounded on $ [0,+\infty) $.
\end{cor}
\begin{proof}
	Combining \eqref{varphint}, \eqref{4pi},  and Proposition \ref{eyj}, we deduce that
	\begin{align*}
		\oint_C\varphi_{\xi^n}^2(\xi,t)d\xi &\le  cM\int_0^tEdt+\oint_C\varphi_{\xi^n}^2(\xi,0)d\xi\\
		& \le 4\pi cM+\oint_C\varphi_{\xi^n}^2(\xi,0)d\xi,
	\end{align*}
	where $ M $ is the upper bound of $ E^{2n+2}+1 $. This means that  $\displaystyle\oint_C\varphi_{\xi^n}^2(\xi,t)d\xi$ has a uniform  upper bound.
	On the other hand, it is evident  that $\displaystyle \oint_C\varphi^2_{\xi^n}d\xi\geq0$, and hence we complete the proof.
\end{proof}

We derive the  following proposition, which implies that $ \displaystyle\frac{d}{dt}E $ is uniformly bounded on $ [0,+\infty) $.
\begin{prop}\label{dEyj}
For every nonnegative integer $ n $, $\displaystyle \frac{d}{dt}\oint_C\varphi_{\xi^n}^2d\xi $ is uniformly bounded on $ [0,+\infty) $.
\end{prop}
\begin{proof}
First,  from \eqref{varphint} and Proposition \ref{eyj}, it is evident  that $\displaystyle \frac{d}{dt}\oint_C\varphi_{\xi^n}^2d\xi$ has a uniform upper bound on $ [0,+\infty) $.
On the other hand, utilizing the related results in  the proof of Proposition \ref{nee},  we deduce that
	\begin{align*}
		\frac{d}{dt}\oint_C\varphi_{\xi^n}^2d\xi&=-\oint_C\varphi_{\xi^{n+1}}^2d\xi\\
        &\qquad\quad+\sum_{i_1+i_2+i_3=n}^{}\oint_C\varphi_{\xi^{i_1}}*\varphi_{\xi^{i_2}}*\varphi_{\xi^{i_3}}*\varphi_{\xi^n}d\xi+4\oint_C\varphi_{\xi^n}^2d\xi\\
		&\geq -\oint_C\varphi_{\xi^{n+1}}^2d\xi-K_1\sum_{i_1+i_2+i_3=n}^{}\oint_C\left|\varphi_{\xi^{i_1}}\varphi_{\xi^{i_2}}\varphi_{\xi^{i_3}}\varphi_{\xi^n}\right| d\xi\\
		&\geq-(1+K_1K_2\epsilon_1)\oint_C\varphi_{\xi^{n+1}}^2d\xi\\&\quad-\frac{K_1K_2}{2n+2}c_1^{2n+2}\left(\frac{2n+2}{2n+1}\epsilon_1\right)^{-(2n+1)}E^{2n+3}.
	\end{align*}
	Therefore, employing Proposition \ref{eyj} and Corollary \ref{cor-vn2-bd}, it follows that  $\displaystyle \frac{d}{dt}\oint_C\varphi_{\xi^n}^2d\xi$ has a uniform lower bound. This completes the proof.
\end{proof}

The following result shows that  the energy becomes vanishingly small as time goes to infinity.

\begin{prop}\label{prop-Eto0}
	Let $ C(\cdot,t) $ be a solution to the flow \eqref{pC_t}. Then its energy  $ E(t) $ converges to zero as time tends to infinity, that is,
	\begin{equation*}
		\displaystyle\lim\limits_{t\rightarrow+\infty}E(t)=0.
	\end{equation*}
\end{prop}
\begin{proof}
	First Proposition \ref{dEyj} allows us to set $ \displaystyle\left|\frac{d}{dt}E(t)\right|\le M  $ for $ t\in [0,+\infty) $. Then it is easy to see  that $ E(t) $ is uniformly continuous on $ [0,+\infty) $. Given  $ \epsilon >0 $, and we can choose  $ \displaystyle\delta=\frac{\epsilon}{M} $
  	so that
	\begin{equation*}
		\left| E(t_1)-E(t_2)\right|\le M\left| t_1-t_2\right|<\epsilon \qquad \text{if}\qquad  \left| t_1 -t_2\right| <\delta.
	\end{equation*}
	Since  $  \displaystyle\int_{0}^{+\infty}E(\tau)d\tau $ converges according to Lemma \ref{lem-0toinfty},  there exists $ T \ge 0$ such that
	\begin{equation*}
		\int_{T_1}^{T_2}E(\tau)d\tau <\frac{\delta^2}{2} \qquad \text{for all}\qquad T_2>T_1 \ge T .
	\end{equation*}
  Hence for every $ t>T $, there exists $ T_3 \ge T$ such that $  t\in\left[T_3,T_3+\frac{\delta}{2}\right] $,  it follows that
	\begin{align*}
		E(t)& =\left| E(t)-\frac{2}{\delta}\int_{T_3}^{T_3+\frac{\delta}{2}}E(\tau)d\tau+\frac{2}{\delta}\int_{T_3}^{T_3+\frac{\delta}{2}}E(\tau)d\tau\right|
		\\ &\le \frac{2}{\delta}\int_{T_3}^{T_3+\frac{\delta}{2}}\left| E(t)-E(\tau)\right| d\tau+\frac{2}{\delta}\left| \int_{T_3}^{T_3+\frac{\delta}{2}}E(\tau)d\tau\right| \\
		& <\epsilon+\delta\\ & =\left(1+\frac{1}{M}\right)\epsilon,
	\end{align*}
	which implies that $ E(t) $ converges to zero as $ t\rightarrow+\infty $.
\end{proof}

We may now complete the proof of Theorem \ref{thm}.

\begin{proof}[Proof of Theorem \ref{thm}]
	The result of Corollary \ref{cor-vn2-bd}, together with Sobolev inequality, show that all derivatives of $\displaystyle \varphi(\xi,t) $ are uniformly bounded. Utilizing \eqref{huanwei} and the evolution equation \eqref{phi_t}, we can bound the time derivatives also. Thus, according to the  Arzela-Ascoli theorem, for any sequence of times $\displaystyle \{t_k\} $, there exists a subsequence $ \displaystyle \{t_{k_i}\} $  such that
	$ \displaystyle\varphi(\xi,t_{k_i}) $ converges smoothly to a function $ f(\xi) $ as $ \displaystyle t_{k_i}\rightarrow +\infty $. In view of Proposition \ref{prop-Eto0}, $ f(\xi) $ must be zero, i.e., $ f(\xi)\equiv0 $. Since every subsequence converges to zero, $ \varphi $ converges smoothly to zero as $ t\rightarrow+\infty $. Hence we complete the proof.
\end{proof}

\section{The backward in time limit}
In this section, we will discuss the backward limit of the centro-affine curvature of  the flow \eqref{pC_t}.
Assume from now on that $ C(\cdot,t) $ is a solution to the flow \eqref{pC_t} defined on $ (-\infty,+\infty) $ and its centro-affine curvature $ \varphi(\xi,t) $ is uniformly bounded on $ S^1\times(-\infty,+\infty) $.
As a simple example, consider
\begin{exmp}\label{exm-el}
	The ellipse is a static solution of the flow \eq{pC_t}. In the Euclidean setting (refer to \cite{qy} for details), the flow \eqref{pC_t} is equivalent to
	\begin{equation*}
 C_t=-\frac{1}{2}\log\left(\frac{\kappa}{h^3}\right)C, \qquad C(\cdot,0)=C_0(\cdot).
	\end{equation*}
 If the initial curve is an ellipse expressed by $(a_0\cos\theta,~b_0\sin\theta)^{\mathrm{T}}$, then the solution is
	\begin{equation*}
		C(\theta,t)=(a_0b_0)^{(\exp(2t)-1)/2}(a_0\cos\theta,~b_0\sin\theta)^{\mathrm{T}}.
	\end{equation*}
 Hence the ellipse satisfies that $a_0b_0=1$ is the static solution.
	Note that $C(\theta, t)$ converges to the ellipse with area equal to $\pi$ as $t\to-\infty$.
\end{exmp}

We will first obtain derivative estimates for the curvature by standard bootstrapping argument.
\begin{prop}\label{ade}
	For every integer $ n \ge 0 $ there exists a constant $ C_n $ such that $ \left|\varphi_{\xi^n}(\xi,t)\right|\le C_n$ on $ S^1\times(-\infty,+\infty). $
\end{prop}
	\begin{proof}
		The proof proceeds by induction on $ n $.
		According to our assumptions, the result  evidently holds for $ n=0 $. Assume that it holds up to $ n-1 $, namely $  \left|\varphi_{\xi^m}\right|\le C_m$ for all $ 0\le m\le n-1 $, then we prove that it also holds for $ n $.
		First of all, we rescale the time variable $ \displaystyle\tau=\frac{1}{2}t $  and  employ the notational convention in Section \ref{EE},  then \eqref{g_t}, \eqref{phi_t}, and \eqref{xi_t} can be rewritten in the form
		\begin{align*}\label{phi_tau}
			\nonumber\frac{g_\tau}{g}&=\varphi^2,\\
			\displaystyle\varphi_\tau&=\varphi_{\xi\xi}+\varphi*\varphi*\varphi+4\varphi,
		\end{align*}
		and
		\begin{equation*}\label{phi_xitau}
			\varphi_{\xi\tau}=\varphi_{\xi^3}+\varphi*\varphi*\varphi_\xi+4\varphi_\xi.
		\end{equation*}
		More generally, Lemma \ref{xint} yields
		\begin{equation*}
			\varphi_{\xi^n\tau}=\varphi_{\xi^{n+2}}+\sum_{i_1+i_2+i_3=n}^{}\varphi_{\xi^{i_1}}*\varphi_{\xi^{i_2}}*\varphi_{\xi^{i_3}}+4\varphi_{\xi^n}.
		\end{equation*}
		Assume $ \tau_0\in\mathbb{R} $ is given,  and we  put $\displaystyle \tilde{\tau}=\tau-\tau_0 $ for $\displaystyle \tau\in[\tau_0,\tau_0+1]$ so that $ \tilde{\tau}\in [0,1]. $  Now let us  consider the key  quantity
		$\displaystyle\tilde{\tau}\varphi_{\xi^n}^2+M\varphi_{\xi^{n-1}}^2 $, where $ M $ is a positive constant to be determined later, and compute its   evolution equation.
		A direct computation  demonstrates that
		\begin{align*}
			\left(\varphi_{\xi^{n-1}}^2\right)_\tau=&\left(\varphi_{\xi^{n-1}}^2\right)_{\xi\xi}-2\varphi_{\xi^n
			}^2\\ &+\sum_{j_1+j_2+j_3=n-1}^{}\varphi_{\xi^{j_1}}*\varphi_{\xi^{j_2}}*\varphi_{\xi^{j_3}}*\varphi_{\xi^{n-1}}+8\varphi_{\xi^{n-1}}^2
		\end{align*}
		and
		\begin{align*}
			\left(\varphi_{\xi^{n}}^2\right)_\tau=&\left(\varphi_{\xi^{n}}^2\right)_{\xi\xi}-2\varphi_{\xi^{n+1}
			}^2\\ &+\sum_{i_1+i_2+i_3=n}^{}\varphi_{\xi^{i_1}}*\varphi_{\xi^{i_2}}*\varphi_{\xi^{i_3}}*\varphi_{\xi^{n}}+8\varphi_{\xi^{n}}^2.
		\end{align*}
		Then the evolution equation for the quantity follows
		\begin{align*}
			\left(\tilde{\tau}\varphi_{\xi^n}^2+M\varphi_{\xi^{n-1}}^2\right)_\tau=\;&\varphi_{\xi^n}^2+\tilde{\tau}\bigg(\left(\varphi_{\xi^n}^2\right)_{\xi\xi}-2\varphi_{\xi^{n+1}}^2\\ &+\sum_{i_1+i_2+i_3=n}^{}\varphi_{\xi^{i_1}}*\varphi_{\xi^{i_2}}*\varphi_{\xi^{i_3}}*\varphi_{\xi^{n}}+8\varphi_{\xi^{n}}^2\bigg)\\
			&+M\bigg(\left(\varphi_{\xi^{n-1}}^2\right)_{\xi\xi}-2\varphi_{\xi^n
			}^2 \\ &+\sum_{j_1+j_2+j_3=n-1}^{}\varphi_{\xi^{j_1}}*\varphi_{\xi^{j_2}}*\varphi_{\xi^{j_3}}*\varphi_{\xi^{n-1}}+8\varphi_{\xi^{n-1}}^2\bigg).
		\end{align*}
		Let us next  insert a discussion about the sum  $\displaystyle \sum\limits_{i_1+i_2+i_3=n}\varphi_{\xi^{i_1}}*\varphi_{\xi^{i_2}}*\varphi_{\xi^{i_3}}*\varphi_{\xi^{n}} $. For the item $ \varphi*\varphi*\varphi_{\xi^n}*\varphi_{\xi^n} $ among them, we have the estimate
		\begin{equation*}
			\varphi*\varphi*\varphi_{\xi^n}*\varphi_{\xi^n} \le K_1C_0^2\varphi_{\xi^n}^2,
		\end{equation*}
		where $ K_1$ is some positive constant.
		For the rest, applying Young's inequality and induction hypothesis, we deduce that
		\begin{equation*}
			\sum_{\substack{i_3\neq n\\ i_1+i_2+i_3=n}}\varphi_{\xi^{i_1}}*\varphi_{\xi^{i_2}}*\varphi_{\xi^{i_3}}*\varphi_{\xi^{n}}\le K_2\left(\frac{\varphi_{\xi^n}^2}{2}+\frac{1}{2}\right)\sum_{\substack{i_3\neq n\\ i_1+i_2+i_3=n}}C_{i_1}C_{i_2}C_{i_3},
		\end{equation*}
		where $ K_2 $ is some positive constant.
		Combining the above discussion and noting  that $ \tilde{\tau}\le1 $, we arrive at
		\begin{align*}
			\left(\tilde{\tau}\varphi_{\xi^n}^2+M\varphi_{\xi^{n-1}}^2\right)_\tau\le&\left(\tilde{\tau}\varphi_{\xi^n}^2+M\varphi_{\xi^{n-1}}^2\right)_{\xi\xi}\\& +\left(K_1C_0^2+9+\frac{K_2}{2}\sum_{\substack{i_3\neq n\\ i_1+i_2+i_3=n}}C_{i_1}C_{i_2}C_{i_3}-2M\right)\varphi_{\xi^n}^2\\&+\frac{K_2}{2}\sum_{\substack{i_3\neq n\\ i_1+i_2+i_3=n}}^{}C_{i_1}C_{i_2}C_{i_3}+8MC_{n-1}^2\\&+K_3M\sum_{j_1+j_2+j_3=n-1}^{}C_{j_1}C_{j_2}C_{j_3}C_{n-1},
		\end{align*}
		where $ K_3 $ is some positive constant.
		For brevity here, we will use the notation $ K $ to replace the last three terms, that is to say,
		\begin{equation*}\label{K}
			\begin{aligned}
				K=&\frac{K_2}{2}\sum_{\substack{i_3\neq n\\ i_1+i_2+i_3=n}}^{}C_{i_1}C_{i_2}C_{i_3}+8MC_{n-1}^2\\&+K_3M\sum_{j_1+j_2+j_3=n-1}^{}C_{j_1}C_{j_2}C_{j_3}C_{n-1}.
			\end{aligned}
		\end{equation*}
		Therefore, choosing an appropriate constant $ M $ such that
		\begin{equation*}
			K_1C_0^2+9+\frac{K_2}{2}\sum_{\substack{i_3\neq n\\ i_1+i_2+i_3=n}}C_{i_1}C_{i_2}C_{i_3}-2M\le0
		\end{equation*}
		leads to
		\begin{align*}
			\left(\tilde{\tau}\varphi_{\xi^n}^2+M\varphi_{\xi^{n-1}}^2\right)_\tau\le&\left(\tilde{\tau}\varphi_{\xi^n}^2+M\varphi_{\xi^{n-1}}^2\right)_{\xi\xi}+K,
		\end{align*}
		which implies the quantity  $ 	\tilde{\tau}\varphi_{\xi^n}^2+M\varphi_{\xi^{n-1}}^2 $ is a subsolution of the semilinear heat equation
		\begin{align}\label{semil-hq}
			\frac{\partial u}{\partial \tau}=&\frac{\partial^2u}{\partial\xi^2}+K.
		\end{align}
		It is worth pointing out that $ 	\tilde{\tau}\varphi_{\xi^n}^2+M\varphi_{\xi^{n-1}}^2\le MC_{n-1}^2 $ at $ \tau=\tau_0 $, so we can adapt the scalar  maximum principle  (see \cite{bcdk}) to \eqref{semil-hq}.
		Let $ v $ be the solution to the  associate initial value problem
		\begin{equation*}
			\frac{dv}{d\tau}=K \qquad \text{satisfying} \qquad v(\tau_0)=MC_{n-1}^2.
		\end{equation*}
		Then the maximum principle guarantees that, for  all $ (\xi,\tau)\in S^1\times[\tau_0,\tau_0+1]  $,  $	\tilde{\tau}\varphi_{\xi^n}^2+M\varphi_{\xi^{n-1}}^2 \le v(\tau)  $, that is
		\begin{equation*}
			(\tau-\tau_0)\varphi_{\xi^n}^2+M\varphi_{\xi^{n-1}}^2\le K(\tau-\tau_0)+MC_{n-1}^2.
		\end{equation*}
		Hence at $ \tau=\tau_0+1 $, there holds
		\begin{equation*}
			\varphi_{\xi^n}^2(\xi,\tau_0+1)+M\varphi_{\xi^{n-1}}^2\le K+MC_{n-1}^2.
		\end{equation*}
		It follows that
		\begin{equation}\label{gj_k}
			\varphi_{\xi^n}^2(\xi,\tau_0+1)\le K+MC_{n-1}^2.
		\end{equation}
		Recalling $ \tau_0\in\mathbb{R} $ was arbitrary, we then obtain $\displaystyle \left|\varphi_{\xi^n}(\xi,\tau)\right|\le\sqrt{K+MC_{n+1}^2}$ on $ S^1\times(-\infty,+\infty) $.
		Note that $ K $, $ M $, and $ C_{n+1} $ are all constants. Thus we complete the proof.
\end{proof}
\begin{rem}\label{as-all}
	From Lemma  \ref{xint} and the above proposition, we know that all spatial and time derivatives of $ \varphi $ are uniformly bounded on $ S^1\times(-\infty,+\infty) $.
\end{rem}
\begin{cor}\label{as-dbound}
	For every integer $ n\ge0 $, $\displaystyle\frac{d}{dt}\oint_{C}\varphi^2_{\xi^n}(\xi,t)d\xi$ is uniformly bounded on the interval $(-\infty,+\infty)  $.
\end{cor}
\begin{proof}
	A direct computation shows that
	\begin{align*}
		\left|\frac{d}{dt}\oint_C\varphi_{\xi^n}^2d\xi\right|&=\left|2\oint_C\varphi_{\xi^n}\varphi_{\xi^nt}d\xi+\frac{1}{2}\oint_C\varphi^2\varphi_{\xi^n}^2d\xi\right|\\&\le2\oint_C\left|\varphi_{\xi^n}\varphi_{\xi^nt}\right| d\xi+\frac{1}{2}\oint_C\left|\varphi^2\varphi_{\xi^n}^2\right| d\xi.
	\end{align*}
	In view of Proposition \ref{dz} and Remark \ref{as-all}, the result readily follows.
\end{proof}
We next deduce the following auxiliary lemma analog to Lemma \ref{lem-0toinfty}. For convenience, the notational convention in Section \ref{EE} is employed here.
\begin{lem}\label{-+c}
	$\displaystyle\int_{-\infty}^{+\infty}E(t)dt$ converges.
\end{lem}
\begin{proof}
	From Lemma \ref{lem-0toinfty}, it is suffice to prove $\displaystyle \int_{-\infty}^{0}E(t)dt $  converges.
	The proof is similar to that of  Lemma \ref{lem-0toinfty}.
	It is not hard to verify that
	\begin{equation*}
		\int_{t}^{0} E(\tau)d\tau=2L(0)-2L(t)\le 2L(0)\le4\pi,
	\end{equation*}
	which implies
	\begin{equation*}
		\int_{-\infty}^{0}E(t)dt\le 4\pi.
	\end{equation*}
	Since $ E(t)\ge0 $,  $  \displaystyle\int_{t}^{0} E(\tau)d\tau$ monotonically increasing backward and has an upper bound. Thus $ 		\displaystyle\int_{-\infty}^{0}E(t)dt$ converges.	
\end{proof}
As a consequence of Corollary \ref{as-dbound} and Lemma \ref{-+c}, the following proposition can be obtained with a similar argument in Proposition \ref{prop-Eto0}, so we list it without proof.
\begin{prop}\label{ae-0}
	Let $ C(\cdot,t) $ be a solution to the flow \eqref{pC_t} defined on $ (-\infty,+\infty)$. Then
	\begin{equation*}
		\displaystyle\lim\limits_{t\rightarrow-\infty}E(t)=0.
	\end{equation*}
\end{prop}

\section{Proof of Theorem \ref{thm2}}\label{proof-thm}
	Let us now derive the following important proposition, and Theorem \ref{thm2} can be regarded as an immediate consequence.
	\begin{prop}\label{imp-prop}
		For a solution to the flow \eqref{pC_t}, we have for every integer $ n\ge 0 $,
		\begin{equation*}
			\max\limits_{\xi\in S^1}\left|\varphi_{\xi^n}(\xi,t)\right|\rightarrow 0  \quad \text{ as } \quad  t\rightarrow -\infty.
		\end{equation*}
	\end{prop}
	\begin{proof}
		The proof is by  induction on $ n $. We first consider the case $ n=0 $.
		It is worth noting that
		\begin{equation*}
			\max\limits_{\xi\in S^1}\varphi^2(\xi,t)=\left(\max\limits_{\xi\in S^1}\left|\varphi(\xi,t)\right|\right)^2,
		\end{equation*}
		so it suffices  to show that
		\begin{equation*}
			\max\limits_{\xi\in S^1}\varphi^2(\xi,t)\rightarrow 0  \quad \text{ as }  \quad t\rightarrow -\infty.
		\end{equation*}
		Suppose, to get a contradiction, that exists $ \epsilon_0>0 $ and we take a  sequence of functions $\displaystyle \varphi^2(\xi_i,t_i) $, where $ t_i\rightarrow-\infty $, such that $ \displaystyle\varphi^2(\xi_i,t_i)\ge \epsilon_0 $ for all $ i $.  In view of Proposition \ref{ade}, we may set $\displaystyle\left|\varphi\right|\le C_0  $ and $ \displaystyle \left|\varphi_\xi\right|\le C_1 $. Then employing  the mean value theorem yields
		\begin{equation*}
			\varphi^2(\xi,t_i)\ge\epsilon_0-2C_0C_1\left|\xi-\xi_i\right|\ge\frac{\epsilon_0}{2}
		\end{equation*}
		for all $ \xi $ satisfying $ \left|\xi-\xi_i\right|\le\frac{\epsilon_0}{4C_0C_1} $.
		It follows that
		\begin{equation*}
			\oint_{C(\cdot,t_i)}\varphi^2(\xi,t_i)d\xi\ge\oint_{\left[\xi_i-\frac{\epsilon_0}{4C_0C_1},~\xi_i+\frac{\epsilon_0}{4C_0C_1}\right]}\varphi^2(\xi,t_i)d\xi\ge\frac{\epsilon_0^2}{4C_0C_1},
		\end{equation*}
		which contradicts Proposition \ref{ae-0}.
		Assume the result  holds up to $ n-1 $, namely, for every $ \epsilon>0 $, there exists $t_1=t_1(\epsilon) $ such that
		\begin{equation*}
			\left|\varphi_{\xi^m}(\xi,t)\right|\le\max\limits_{\xi \in S^1}\left|\varphi_{\xi^m}(\xi,t)\right|\le\epsilon
		\end{equation*}
		for all $t\in(-\infty,t_1]$ and $ 0\le m\le n-1 $.
		We then prove that it holds for $ n $.
		Let us go back to the process in the proof of Proposition \ref{ade}. Notice that the estimates there can be improved by the induction hypothesis. For instance, we may  set	$ K\le\epsilon $ when $ \epsilon $ is sufficiently small.
		Then it follows from \eqref{gj_k} that
		\begin{equation*}
			\varphi_{\xi^n}^2(\xi,\tau_0+1)\le\epsilon+M\epsilon^2.
		\end{equation*}
	Since $ \tau_0\le \frac{1}{2}t_1$ is arbitrary, we conclude that
		\begin{equation*}
			\left|\varphi_{\xi^n}(\xi,t)\right|\rightarrow0 \quad \text{as} \quad t\rightarrow-\infty.
		\end{equation*}
The result thus immediately follows,
		and the proof of this proposition is completed.
\end{proof}

\end{document}